\documentclass[12pt]{article}
\usepackage{amsmath,amssymb,amsfonts,amsthm}
\begin{document}

\theoremstyle{plain}
\newtheorem{theorem}{Theorem}
\newtheorem{lemma}[theorem]{Lemma}
\newtheorem{cor}[theorem]{Corollary}

\newtheorem{prop}[theorem]{Proposition}

\theoremstyle{definition} 
\newtheorem{definition}{Definition}
\newtheorem{example}[theorem]{Example}
\newtheorem{conjecture}[theorem]{Conjecture}

\theoremstyle{remark} 
\newtheorem{remark}[theorem]{Remark}

\def \Z{{\mathbb Z}}
\def \N{{\mathbb N}}
\def \Q{{\mathbb Q}}
\def \R{{\mathbb R}}

\def\e{{\epsilon}}

\title{The congruence of Wolstenholme and generalized binomial coefficients related to Lucas sequences} 

\author{
{\sc Christian ~Ballot} \\
{D\'epartement de Math\'ematiques et M\'ecanique}\\
{Universit\'e de Caen} \\
{F14032 Caen Cedex, France} \\
{christian.ballot@unicaen.fr}\\}

\date{}

\maketitle
\begin{abstract} Wolstenholme's congruence says that  
$\binom{2p-1}{p-1}\equiv1\pmod{p^3}$ for all primes $p\ge5$.  
Kimball and Webb established an analogue of the congruence of Wolstenholme using Fibonomial 
coefficients. This note answers the question: `Is there a common generalization  
to the Wolstenholme and the Kimball and Webb congruences?'.  
Tinted by a positive answer, valid for all fundamental Lucas sequences, we go up the ladder. 
We give a broad generalization of several congruences such as Ljunggren et al's 
$\binom{kp}{\ell p}\equiv\binom{k}\ell\pmod{p^3}$,  
($p\ge5$), or McIntosh's: $\binom{2p-1}{p-1}\equiv1-p^2\sum_{0<t<p}\frac1{t^2}\pmod{p^5}$,  
($p\ge7$), replacing ordinary binomials by generalized binomial coefficients $\binom{*}{*}_U$,  
where $U=U(P,Q)$ is an arbitrary fundamental Lucas sequence. That is, a sequence which satisfies 
$U_0=0$, $U_1=1$ and $U_{t+2}=PU_{t+1}-QU_t$, for all $t\ge0$.
\end{abstract}

\section{Introduction}
\label{sec:1}

 In 1862 Joseph Wolstenholme \cite{Wo} established a now well-known congruence for binomial 
coefficients, namely 
 
 \begin{theorem}\label{thm:W} Let $p$ be a prime number $\ge5$. Then 
\begin{equation}\label{eq:W}
\binom{2p-1}{p-1}\equiv1\pmod {p^3}.
\end{equation}
\end{theorem}

 Charles Babbage \cite{Bab}, in 1819, had actually shown that congruence (\ref{eq:W}) held modulo $p^2$ for 
all primes $p$ greater than $2$. There is a survey paper \cite{Mes} on the   
numerous generalizations of Theorem \ref{thm:W} discovered in the last 150 years. 
This survey also contains many other related results. 

 We focus first our attention on the sligthly more general congruence  
\begin{equation}\label{eq:W+}
\binom{(k+1)p-1}{p-1}\equiv1\pmod {p^3},
\end{equation} which holds for all primes $p\ge5$ and all nonnegative integers $k$. 
According to the survey \cite{Mes}, congruence (\ref{eq:W+}) was proved in 1900 by Glaisher 
(\cite{Glai1} p. 21, \cite{Glai2} p. 33). 

 Lemma 3 of the paper \cite{KW1}, which we rewrite as a theorem below, is an analogue of (\ref{eq:W+}). 

\begin{theorem} \label{thm:KW} Let $p$ be a prime at least $7$ whose rank of appearance $\rho$ in the Fibonacci 
sequence is equal to $p-\e_p$, where $\e_p$ is $\pm1$. Then for all integers $k\ge0$ 
\begin{equation}\label{eq:Fib}
\binom{(k+1)\rho-1}{\rho-1}_F\equiv\e_p^k\pmod {p^3},
\end{equation}
where the symbol $\binom{*}{*}_F$ stands for the Fibonomial coefficient.
\end{theorem} 

 If $A=(a_n)_{n\ge0}$ is a sequence of complex numbers where $a_0=0$ and all $a_n\not=0$ for $n>0$, 
then one defines, for $m$ and $n$ nonnegative integers, the {\it generalized} binomial coefficient 
\begin{equation}\label{eq:def}
\binom{m}{n}_A=\begin{cases}\frac{a_ma_{m-1}\dots a_{m-n+1}}{a_na_{n-1}\dots a_1},& \text{ if }m\ge n\ge1;\\
1,& \text{ if }n=0;\\
0,& \text{ otherwise.}\end{cases} 
\end{equation} The well written paper \cite{Gou} contains a number of early references about these coefficients 
and investigated several of their general properties. We point out another early reference \cite{Wa},  
not often quoted, in which Ward gives two equivalent criteria that imply the integrality of the generalized 
coefficients $\binom{m}{n}_A$ of a sequence of integers $A$. One of them is that  
$A$ be a {\it strong divisibility sequence}, i.e., one for which $a_{\gcd(m,n)}=\gcd (a_m,a_n)$ for all 
$m>n>0$; the other criterion is expressed in terms of ranks of appearance of prime powers in $A$. The equivalence 
of these two criteria was essentially rediscovered in \cite{KnWi}. 
When $A$ is the Fibonacci sequence 
these binomial coefficients are called {\it Fibonomials} and many papers have studied their properties. Some 
papers have considered the generalized binomial coefficients when $A$ is a fundamental Lucas sequence,  
that is, a sequence $U=U(P,Q)$ satisfying 
\begin{equation}\label{eq:second}
U_0=0,\; U_1=1\;\text{ and }U_{n+2}=PU_{n+1}-QU_n, \;\text{ for all }n\ge0,
\end{equation}
where $(P,Q)$ is a pair of integers, $Q$ nonzero. We will refer to these generalized 
binomials as 
{\it Lucanomial} coefficients in the sequel. Ordinary binomials are Lucanomial 
coefficients with parameters $(P,Q)=(2,1)$, whereas the Fibonomials correspond to $(P,Q)=(1,-1)$.

 Therefore it makes sense to look for a simple congruence for the general Lucanomial     
\begin{equation}
\binom{(k+1)\rho-1}{\rho-1}_U\pmod {p^3},
\end{equation}
valid for an arbitrary Lucas sequence $U$, that would encompass both the congruence (\ref{eq:W+}) and 
Theorem \ref{thm:KW}. 

 Here $\rho$ represents the rank of appearance of the prime $p$ in 
$U$, that is, the least positive integer $t$ such that $p\mid U_t$. It is known to exist 
for all primes $p$ not dividing $Q$ and to divide $p-\e_p$, where $\e_p$ is the Legendre 
character $(D\;|\;p)$ and $D$ is $P^2-4Q$. It is necessary to require, as in Theorem \ref{thm:KW}, 
that the rank $\rho$ be maximal, i.e., be equal to $p-\e_p$. Note that the rank of any 
prime $p$ is maximal and equal to $p$ for $U_n=n$ ($D=0$, $\e_p=0$). However, the case $\e_p=0$  
only occurs for $p=5$ for the Fibonacci sequence $F=U(1,-1)$, a case that Theorem \ref{thm:KW} 
does not address. 
A calculation for $p=5$ yields  
\begin{equation}\label{eq:case0}
\binom{2\rho-1}{\rho-1}_F=\binom{9}{4}_F\equiv1\pmod {125}. 
\end{equation}  This residue of $1$ is at least conform to what one gets in (\ref{eq:W+}), but does not 
match the expression $\e_p^k$ of Theorem \ref{thm:KW} which would yield $0$. 

\medskip

 Thus, one needs to generalize the results of the paper \cite{KW1} from Fibonomial coefficients to 
Lucanomial coefficients and include the case $\e_p=0$ in the analysis. However, some of the results 
leading to Theorem \ref{thm:KW} in \cite{KW1} seem, at first sight, to depend on idiosyncracies 
of the Fibonacci sequence. Thus, a few numerical calculations helped us believe in the  
existence of a generalization and were useful in guiding us to it.

\begin{theorem} \label{thm:N} Let $U=U(P,Q)$ be a fundamental Lucas sequence with parameters $P$ and 
$Q$. Let $p\ge5$, $p\nmid Q$, be a prime whose rank of appearance $\rho$ in $U$   
is equal to $p-\e_p$, where $\e_p$ is the Legendre character $(D\;|\;p)$, $D=P^2-4Q$. 
Then for all integers $k\ge0$ 
\begin{equation}\label{eq:Luc}
\binom{(k+1)\rho-1}{\rho-1}_U\equiv(-1)^{k\e_p}Q^{k\rho(\rho-1)/2}\pmod {p^3},
\end{equation}
where the symbol $\binom{*}{*}_U$ stands for the Lucanomial coefficient.
\end{theorem}

\begin{remark}\label{rem:pow} Theorem \ref{thm:N} implies that for all $k\ge0$ 
$$
\binom{(k+1)\rho-1}{\rho-1}_U\equiv\binom{2\rho-1}{\rho-1}_U^k\pmod{p^3}.
$$
\end{remark}

\begin{remark}\label{rem:kw} Congruence (\ref{eq:W+}), Theorem \ref{thm:KW} and, as readily checked, 
congruence (\ref{eq:case0}) are implied by 
Theorem \ref{thm:N}. Indeed, the sequence $a_n=n$ is $U_n(2,1)$, for which $Q=1$ and $\e_p=0$ 
for all primes.  
To see that Theorem \ref{thm:KW} is a corollary of Theorem \ref{thm:N}, it suffices 
to check that 
$$
\e_p=-(-1)^{\rho(\rho-1)/2},
$$ for every odd prime $p>5$ of maximal rank in 
the Fibonacci sequence $U(1,-1)$. All primes of rank $p\pm1$ in the Fibonacci sequence 
must be congruent to $3\pmod 4$, since by Euler's criterion for Lucas sequences (\ref{eq:criterion}) 
we need 
to have $(-1\;|\;p)=-1$. If $\e_p=1$, that is, if $\rho=p-1$, then $\rho(\rho-1)\equiv2\pmod4$ 
so that $-(-1)^{\rho(\rho-1)/2}=+1=\e_p$. If $\e_p=-1$, that is, $\rho=p+1$, 
then $\rho(\rho-1)\equiv0\pmod4$ so $-(-1)^{\rho(\rho-1)/2}=-1=\e_p$. 
\end{remark} 

 Section \ref{sec:2} of the paper is devoted to some relevant additional remarks  
on Lucas sequences, some useful lemmas and to a proof of Theorem \ref{thm:N}. 

\medskip

 For all primes $p\ge5$ and all nonnegative integers $k$ and $\ell$, we have the congruence 
\begin{equation}\label{eq:Viggo}
\binom{kp}{\ell p}\equiv\binom{k}{\ell}\pmod{p^3}.
\end{equation} This congruence supersedes congruence (\ref{eq:W+}) and was first proved in a 
collective paper \cite{Bru} which 
appeared in 1952. It was reproved by Bailey some 30 years later in the paper \cite{Bai}, where 
the case $(k,\ell)=(2,1)$, which is equivalent to Wolstenholme's congruence (\ref{eq:W}), 
is proved first before an induction on $k$ yielded congruence (\ref{eq:W+}) and 
another proof by induction gave (\ref{eq:Viggo}). Interestingly another simple 
argument, combinatorial, reduces the proof of (\ref{eq:Viggo}) to that of the 
case $(k,\ell)=(2,1)$ in the book \cite{Stan} (see solution of exercice 1.14 p. 165).  

 Similarly in \cite{KW1}, Theorem \ref{thm:KW} is used by the authors to 
produce an analogue of (\ref{eq:Viggo}) for the Fibonacci sequence $U=F$. 
That is, in our notation, for primes $p\ge7$ of rank $\rho=p-\epsilon_p$, 
where $\epsilon_p=\pm1$, their result (\cite{KW1}, p. 296) states that
\begin{equation}\label{eq:KW2}
\binom{kp}{\ell p}_F\equiv\epsilon_p^{(k-\ell)\ell}\binom{k}{\ell}_{F'}\pmod{p^3},
\end{equation} where $F'_t=F_{\rho t}$ for all $t\ge0$, $k$, $\ell$ are integers satisfying $k\ge\ell\ge1$. 
Section \ref{sec:3} states and proves a congruence, Theorem \ref{thm:LjWe}, for Lucanomials 
$\binom{k\rho}{\ell\rho}_U$ 
$\pmod{p^3}$ that subsumes the congruences (\ref{eq:Viggo}) and (\ref{eq:KW2}).  
Here again the proof of this more general result is easily derived from Theorem \ref{thm:N}. 
We raise in passing the question of the existence of a combinatorial argument that 
would reduce Theorem \ref{thm:LjWe} to the case $(k,\ell)=(2,1)$. Note that 
Lucanomial coefficients were given a combinatorial interpretation in \cite{Ben}. Also a $q$-analogue 
of (\ref{eq:Viggo}) that uses $q$-binomial coefficients was established in 
the paper \cite{Stra}.    

\medskip

 In a fourth section, we selected three congruences for binomials $\binom{2p-1}{p-1}\pmod{p^5}$, namely 
(\ref{eq:i}), (\ref{eq:ii}) and  (\ref{eq:iii}),    
and establish for each a generalization to Lucanomial coefficients $\binom{2\rho-1}{\rho-1}_U\pmod{p^5}$ for 
primes $p\ge7$ of maximal rank $\rho$ in $U$. Not to lengthen an already long introduction we only 
state the example of congruence (\ref{eq:iii}), i.e.,  
$$
\binom{2p-1}{p-1}\equiv1-p^2\sum_{0<t<p}\frac1{t^2}\pmod{p^5},
$$ which generalizes into 
$$
\binom{2\rho-1}{\rho-1}_U\equiv(-1)^{\e_p}Q^{\frac{\rho(\rho-1)}2}\bigg[1-4\frac{U_\rho^2}{V_\rho^2}
\sum_{0<t<\rho}\frac{Q^t}{U_t^2}\bigg]\pmod{p^5},
$$ where $U(P,Q)$ is a fundamental Lucas sequence and $V(P,Q)$ is its companion sequence.   

Note that the condition that $p$ be of maximal rank in $U$ may be viewed as a quadratic 
analogue of Artin's conjecture which gives a positive density (equal to a positive 
rational number times Artin's constant) for the set of primes $p$ for which a given 
$a$ is a primitive root $\pmod p$, when $a$ is a non-square integer and $|a|\ge2$.  
Hooley \cite{Hoo} proved Artin's conjecture conditionally to some generalized Riemann hypotheses. 
So did Roskam (\cite{Ros}, \cite{Ros1}) for the set of primes $p$ for which a fundamental unit 
of a quadratic field has maximal order modulo $(p)$. Thus, given $U(P,Q)$, $Q$ not a square, 
our theorems presumably should also concern sets of primes of positive densities.       

\medskip

 In recent years congruences for ordinary binomials $\binom{2p-1}{p-1}\pmod{p^l}$ 
have been established for larger and larger values of $l$ (see \cite{Mes}, p. 4-6). 
No doubt there must be higher corresponding congruences for Lucanomials. In fact, we end the 
paper with such a congruence modulo ${p^6}$. 
Generalizations of (\ref{eq:i}) are stated in Theorems \ref{thm:pcinq1} and \ref{thm:pcinq4}, 
those of (\ref{eq:ii}) and the above congruence (\ref{eq:iii}) appear in Theorems  
\ref{thm:pcinq2} and \ref{thm:pcinq3} respectively. We added an appendix as a short fifth 
section where the integrality of all Lucanomial coefficients $\binom{m}{n}_U$ is asserted 
for all $U$ Lucas sequences. 

\medskip

 Familiarity with Lucas sequences is assumed throughout the paper, but the reader may want 
to consult the introduction of \cite{Ba3} and the references it mentions. Chapter 4 of the 
book \cite{Wi} is a useful introduction to these sequences. 

\medskip

 Lucanomial coefficients have already been the object of generalizations of 
classical arithmetic properties of ordinary binomial coefficients. 
Kummer's theorem giving the exact power of a prime $p$ 
in the binomial coefficient $\binom{m+n}{n}$ as the number of carries in the addition of $m$ 
and $n$ in radix $p$ was generalized to all strong divisibility 
sequences of positive integers \cite{KnWi}. That includes, in particular, all Lucas sequences $U(P,Q)$ 
with positive terms when $P$ and $Q$ are coprime. 

 Also a generalization of the celebrated theorem of Lucas:  
$$
\binom{mp+r}{np+s}\equiv\binom{m}{n}\binom{r}{s}\pmod p,
$$ where $r$ and $s$ are nonnegative integers less than the prime $p$, 
was achieved in terms of Lucanomials $\binom{mp+r}{np+s}_U$, under the 
hypothesis that $U(P,Q)$ is a Lucas sequence with $\gcd(P,Q)=1$, $P\not=0$ and 
$P^2Q\not=1$ (see \cite{HuSu}). 

 In fact both the theorems of Kummer and of Lucas had been generalized in an earlier paper 
\cite{Fr} but with respect to $q$-binomial coefficients.      
 
\section{Preliminaries and a proof of Theorem \ref{thm:N}}
\label{sec:2}

 Lucas theory is often developped with the two hypotheses that  $U(P,Q)$ is  
nondegenerate and $\gcd(P,Q)$ is $1$. 
The Lucas sequence $U(P,Q)$ is called {\it degenerate} whenever the ratio of the 
zeros $\alpha$ and $\beta$ of $x^2-Px+Q$ is a root of unity. We do not make any of these assumptions 
here. If $U$ is degenerate then we must have $U_2U_3U_4U_6=0$. Indeed, if $\alpha\not=\beta$ 
then $U_t=\frac{\alpha^t-\beta^t}{\alpha-\beta}$ and the ratio $\alpha/\beta$, lying in the 
quadratic field $\mathbb Q(\sqrt{D})$, must be a second, third, fourth or sixth root of unity. 
Thus, some terms of the sequence $U$ will be $0$, but rather than discard those Lucas sequences 
from our analysis, we make a small amendment 
to the definition (\ref{eq:def}) to ensure that the corresponding 
Lucanomials $\binom{m}{n}_U$ are well defined as rational numbers. Although 
the hypotheses of Theorems \ref{thm:N}, \ref{thm:LjWe} or of the theorems of Section \ref{sec:4} 
if applied to a prime $p\ge11$ prevent the corresponding Lucanomials from having zero terms, 
this is not necessarily the case if $p=5$ or $p=7$.  
With $\gcd(P,Q)>1$, the 
Lucas sequence $A=U(P,Q)$ is no longer a strong divisibility sequence. Nevertheless $A$, or $\lambda A$, 
$\lambda$ an integer, satisfies some `convexity' property. Namely for all prime powers 
$p^a$ ($a\ge1$), $p\nmid2Q$, and for all $x\ge2$, we have  
\begin{equation}\label{eq:convex}
\#\;\{t\in[x],\;p^a\mid A_t\}\;\ge\;\#\;\{t\in[y],\;p^a\mid A_t\}+\#\;\{t\in[x-y],\;p^a\mid A_t\},
\end{equation} for all $y\in[x-1]$. Here, if $z$ is an integer $\ge1$, $[z]$ denotes the 
set of natural numbers $1,2,\hdots,z$. This property holds because for such prime powers $p^a$, 
we have $p^a\mid U_t$ iff $\rho(p^a)\mid t$, where $\rho(p^a)$ is the rank of appearance 
of $p^a$ in $U$, and because $\lfloor x+y\rfloor\ge\lfloor x\rfloor+\lfloor y\rfloor$ for 
all real numbers $x$ and $y$. 

\medskip 

 The convention we adopt for the generalized binomials $\binom{m}{n}_A$ of definition (\ref{eq:def}) 
is that if there are zero terms in the product $\prod_{i=1}^n\frac{a_{m+1-i}}{a_i}$ then 
\begin{equation}\label{eq:conv}
\text{a }0\text{ in the numerator and a }0\text{ in the denominator 
{\bf cancel out} as a }1.
\end{equation} 

 With convention (\ref{eq:conv}), property (\ref{eq:convex}) satisfied by $A=\lambda U$, for all Lucas sequences $U$, 
guarantees that the generalized binomial $\binom{m}{n}_A$ is a well defined rational number. 
Indeed this property implies that the number of $0$ terms in the numerator of  
$\prod_{i=1}^n\frac{a_{m+1-i}}{a_i}$ is at least that of its denominator. It also implies that $\binom{m}{n}_A$, $m$ 
and $n$ nonnegative integers, is 
well defined $p$-adically for all primes $p\nmid2Q$. In fact we can show it is always a rational integer.
\footnote{See our short Appendix}      

\medskip

 To each fundamental Lucas sequence $U(P,Q)$ we associate a {\it companion} Lucas sequence $V=V(P,Q)$ 
which obeys recursion (\ref{eq:second}), but has initial values $V_0=2$ and $V_1=P$. 
The following identities are all classical ones and are all valid no matter what the value  
of $\gcd(P,Q)$ is. We will use them throughout the paper.   

\begin{eqnarray}
2U_{s+t} & = & U_sV_t+U_tV_s,\label{eq:1}\\
2V_{s+t} & = & V_sV_t+DU_sU_t,\label{eq:2}\\
V_t^2-DU_t^2 & = & 4Q^t,\label{eq:3}\\
U_{2t} & = & U_tV_t,\label{eq:4}\\
V_{2t} & = & V_t^2-2Q^t,\label{eq:5}\\
2Q^tU_{s-t} & = & U_sV_t-U_tV_s.\label{eq:6}
\end{eqnarray}

 We referred to Euler's criterion for Lucas sequences in our introduction. The criterion  
states that 
\begin{equation}\label{eq:criterion}p\mid U_{(p-\epsilon_p)/2} \;\text{ iff }\;
Q \text{ is a square modulo }p,
\end{equation} where $U(P,Q)$ is a fundamental Lucas sequence and $p$ is a prime that 
does not divide $2DQ$ (see \cite{Wi}, pp. 84--85). 

 Note that our theorems and the lemmas of Section \ref{sec:4} all deal with primes $p\ge5$ 
of maximal rank. In their statements, we sometimes  
omit to mention the condition $p\nmid Q$, because that condition is necessary. 
Indeed, if $p\mid Q$, then, by (\ref{eq:second}), $U_t\equiv P^{t-1}\pmod p$. Thus, $p$ has 
no rank, because if $p$ divided $P$, then $\rho(p)$ would be equal to $2$, as $U_2=P$, a 
contradiction.   

\medskip

 Given a prime $p$ of rank $\rho$ and a nonnegative integer $\nu$, we write 
\begin{equation}\label{eq:nota}
\Sigma_\nu:=\sum_{0<t<\rho}\frac{V_t^\nu}{U_t^\nu} 
\text{ and }\Sigma_{1,1}:=\sum_{0<s<t<\rho}\frac{V_sV_t}{U_sU_t}.
\end{equation}

 The proof of Theorem \ref{thm:N} we are about to write uses a few lemmas which we  
state first. 

\begin{lemma} \label{lem:psquare} Let $(U,V)$ be a pair of Lucas 
sequences with parameters $P$ and $Q$. Let $\nu$ be a nonnegative integer. 
If $p\nmid Q$ is a prime at least $\nu+3$ of maximal rank 
$\rho$, i.e., of rank $p-\e_p$, where $\e_p=0$ or $\pm1$, then 
\footnote{less $\nu=0$ and $\epsilon_p=0$ when $\Sigma_0\equiv-1\pmod p$} 
\begin{equation}\label{eq:lot}
\Sigma_\nu\equiv\begin{cases}0\pmod {p^2}, &\text{ if }\nu\text{ is odd };\\
0\pmod p, &\text{ if }\epsilon_p=-1\text{ or }0;\\
-2D^{\nu/2}\pmod p, & \text{if }\nu\text{ is even and }\epsilon_p=1.\end{cases}
\end{equation} Moreover, if $p$ is an odd prime not dividing $Q$ of rank $\rho$, then 
\begin{equation}\label{eq:+}
\Sigma_\nu\equiv0\pmod p, \quad\text{ when }\nu \text{ is odd.}
\end{equation}  
\end{lemma} 
\begin{proof} The case $\nu$ odd of (\ref{eq:lot}) is Theorem 3 of \cite{Ba1}. 
(The case $\nu=1$ first appeared, nearly complete, as the main theorem of the paper 
\cite{KW2}, but also (nearly) as a corollary of the main theorem of \cite{Pan}, and as a particular 
case of Theorem 4.1 of \cite{Ba2}, or of Theorems 3 and 12 of \cite{Ba3}.)

 The case $\nu$ even can be treated with the very same arguments used in the last part of the proof 
of Theorem 4, p. 5, of \cite{Ba1}. (The basic facts, noted first in \cite{KW2}, are that, by (\ref{eq:6}), 
all $V_t/U_t$ are distinct $\pmod p$ for $t\in(0,\rho)$ and no $V_t/U_t$ is $\pm\sqrt{D}\pmod p$ 
by (\ref{eq:3}); also $p\mid\sum_{t=1}^pt^e$ if $p-1\nmid e$). The condition 
$p\ge\nu+3$ is a sufficient condition which guarantees that $p-1\nmid\nu$ for $\nu\ge2$ even. 

 The additional congruence (\ref{eq:+}) for $\nu$ odd, but without the restrictions that $\rho$ 
be maximal and $p\ge\nu+3$, 
is a consequence of the congruence $\pmod{p^2}$ on the sixth line of the proof of Theorem 4 of \cite{Ba1}.   
\end{proof}    

\begin{lemma} \label{lem:Sig} Let $U=U(P,Q)$ be a fundamental Lucas sequence. If $p\nmid 6Q$ is a prime    
of maximal rank $\rho$ in $U$, then 
$$
\Sigma_{1,1}\equiv\begin{cases} 0\pmod p,& \text{ if }\e_p=0
\text{ or }-1;\\
D\pmod p,& \text{ if }\e_p=1.\end{cases}
$$
\end{lemma}
\begin{proof} We have $\Sigma_1^2=\Sigma_2+2\Sigma_{1,1}$ so that $\Sigma_{1,1}\equiv-\frac12\Sigma_2
\pmod p$, since, by Lemma \ref{lem:psquare}, $p^4$ divides $\Sigma_1^2$ and $\Sigma_2$ is either $0$ 
or $-2D\pmod p$. 
\end{proof}

\begin{lemma} \label{lem:Vro} Let $U=U(P,Q)$ be a fundamental Lucas sequence. If $p\nmid Q$ is an odd prime of even 
rank $\rho$ in $U$ and $k\ge1$ is an odd integer, then 
$$
\frac{V_{k\rho}}{2}\equiv-Q^{k\rho/2}\pmod {p^2}.
$$
\end{lemma}
\begin{proof} Since $p$ divides $U_{k\rho}$, but not $U_{k\rho/2}$, we find by 
(\ref{eq:4}) that $p$ divides $V_{k\rho/2}$. Therefore, from 
(\ref{eq:5}) with $t=k\rho/2$, we deduce that $V_{k\rho}\equiv-2Q^{k\rho/2}\pmod {p^2}$. 
\end{proof} 

\begin{lemma} \label{lem:V2} Let $V=V(P,Q)$ be a companion Lucas sequence. Let $m$ be an integer $\ge2$. 
Suppose $V_t\equiv\pm2Q^{t/2}\pmod m$. Then 
$$
V_{2t}\equiv2Q^t\pmod m.
$$
\end{lemma}
\begin{proof} We have $V_{2t}=V_t^2-2Q^t\equiv2Q^t\pmod m$.
\end{proof}

 We are now ready for a proof of Theorem \ref{thm:N}. 

\begin{proof} We have 
$$
\binom{(k+1)\rho-1}{\rho-1}_U=\frac{\prod_{t=1}^{\rho-1}U_{k\rho+t}}{\prod_{t=1}^{\rho-1}U_t}.
$$
By the addition formula (\ref{eq:1}), we find that 
\begin{eqnarray*}
2^{\rho-1}\prod_{t=1}^{\rho-1}U_{k\rho+t}&=&\prod_{t=1}^{\rho-1}(V_{k\rho} U_t+U_{k\rho} V_t)\\
&\equiv&(V_{k\rho}^{\rho-1}+V_{k\rho}^{\rho-2}U_{k\rho}\Sigma_1+V_{k\rho}^{\rho-3}U_{k\rho}^2
\Sigma_{1,1})\times\prod_{t=1}^{\rho-1}U_t\label{eq:dev}\\
&\equiv&(V_{k\rho}^{\rho-1}+V_{k\rho}^{\rho-3}U_{k\rho}^2
\Sigma_{1,1})\times\prod_{t=1}^{\rho-1}U_t\pmod {p^3},
\end{eqnarray*} since $p$ divides $U_{k\rho}$ and, by Lemma \ref{lem:Sig}, $\Sigma_1$ is $0\pmod {p^2}$. 

 We first examine the cases $\rho$ is $p+1$ and $\rho$ is $p$. In those cases 
$U_{k\rho}^2\Sigma_{1,1}$ is $0\pmod{p^3}$ by Lemma \ref{lem:Sig}. 
Hence, 
$$
\binom{(k+1)\rho-1}{\rho-1}_U\equiv\bigg(\frac{V_{k\rho}}{2}\bigg)^{\rho-1}\pmod{p^3}. 
$$ 
If $\rho$ is $p$, then, by (\ref{eq:3}) and the fact that 
$p^3\mid DU_{k\rho}^2$, we see that $V_{k\rho}^2\equiv4Q^{k\rho}\pmod{p^3}$. 
Therefore, 
$$ 
\binom{(k+1)\rho-1}{\rho-1}_U\equiv (Q^{k\rho})^{(\rho-1)/2}\pmod{p^3},
$$ yielding the result in that case. 
If $\rho$ is $p+1$ and $k$ is odd, then by Lemma \ref{lem:Vro} there is an integer 
$\lambda$ such that $\frac{V_{k\rho}}{2}=-Q^{k\rho/2}+\lambda p^2$. 
Raising members of the previous equation to the $p$th power gives $(V_{k\rho}/2)^p\equiv-Q^{k\rho p/2}
\pmod{p^3}$. But $-1=(-1)^{-k}$ so the theorem follows in that case. 

If $\rho$ is $p+1$ and $k=2^a\ell$, where $\ell$ is odd and $a\ge1$, then, by Lemma \ref{lem:Vro}, 
we have $V_{\ell\rho}\equiv-2Q^{\ell\rho/2} \pmod{p^2}$. Applying $a$ times Lemma \ref{lem:V2} 
we see that $V_{k\rho}\equiv2Q^{k\rho/2}\pmod{p^2}$. As we did in the case $k$ odd, we raise both 
sides of the congruence to the $p$th power to obtain $(V_{k\rho}/2)^p\equiv Q^{k\rho(\rho-1)/2}=
(-1)^{k\e_p}Q^{k\rho(\rho-1)/2}\pmod{p^3}$ and the theorem follows.

Suppose now $\e_p$ is $1$, that is, $\rho$ is $p-1$. By Lemma \ref{lem:Sig}, $\Sigma_{1,1}\equiv D\pmod p$ so 
that $U_{k\rho}^2\Sigma_{1,1}\equiv DU_{k\rho}^2\pmod{p^3}$. But, by (\ref{eq:3}), $DU_{k\rho}^2=
V_{k\rho}^2-4Q^{k\rho}$. Therefore, we have 
$$
2^{\rho-1}\binom{(k+1)\rho-1}{\rho-1}_U\equiv2V_{k\rho}^{\rho-1}-4Q^{k\rho}V_{k\rho}^{\rho-3}\pmod{p^3}.
$$
This gives 
\begin{equation}\label{eq:long}
\binom{(k+1)\rho-1}{\rho-1}_U\equiv\bigg(\frac{V_{k\rho}}{2}\bigg)^p\bigg[2\bigg(\frac2{V_{k\rho}}\bigg)^2
-Q^{k\rho}\bigg(\frac2{V_{k\rho}}\bigg)^4\bigg]\pmod{p^3}.
\end{equation}  
By Lemma \ref{lem:Vro}, we have $V_{k\rho}/2\equiv-Q^{k\rho/2}\pmod {p^2}$ in case $k$ is odd. Using Lemma 
\ref{lem:V2}, as for the case $\rho=p+1$, we get that $V_{k\rho}/2\equiv Q^{k\rho/2}\pmod {p^2}$ if 
$k$ is even. Thus, generally, $V_{k\rho}/2\equiv(-1)^kQ^{k\rho/2}\pmod {p^2}$.  Raising the previous congruence 
to the $p$th power yields $(V_{k\rho}/2)^p\equiv(-1)^kQ^{kp\rho/2}\pmod{p^3}$, while inverting 
it yields the existence of an integer $\mu$ such that $2/V_{k\rho}\equiv(-1)^kQ^{-k\rho/2}+\mu p^2
\pmod{p^3}$. Thus, with $\alpha_{p,k}:=$ the bracket factor of the righthand side of (\ref{eq:long}), we find that modulo $p^3$
\begin{eqnarray*}
\alpha_{p,k}&\equiv&2\big((-1)^kQ^{-k\rho/2}+\mu p^2\big)^2
-Q^{k\rho}\big((-1)^kQ^{-k\rho/2}+\mu p^2\big)^4\\
&\equiv&(2Q^{-k\rho}+(-1)^k4Q^{-k\rho/2}\mu p^2)-Q^{k\rho}(Q^{-2k\rho}+(-1)^k4Q^{-3k\rho/2}\mu p^2)\\
&=&Q^{-k\rho}.
\end{eqnarray*}
Thus, we end up with 
$$
\binom{(k+1)\rho-1}{\rho-1}_U\equiv(-1)^kQ^{kp\rho/2}Q^{-k\rho}=(-1)^{k\e_p}Q^{k\rho(p-2)/2}\pmod{p^3},
$$
which yields the theorem.
\end{proof}

 The above proof is the first that came to us. It proceeds case by case according to whether the 
value of the rank of $p$ is $p+1$, $p$ or $p-1$ and, thus, appears somewhat miraculous. Although we 
initially wrote case by case proofs for the higher congruences of Section \ref{sec:4}, we ended up 
finding a global and more natural approach at least for Theorems \ref{thm:pcinq2} and \ref{thm:pcinq3}.  

\section{Lucanomials $\binom{k\rho}{\ell\rho}_U\pmod{p^3}$}
\label{sec:3}

Here is our common generalization of the Ljunggren et al. congruence (\ref{eq:Viggo}) 
and Kimball and Webb's theorem (\ref{eq:KW2}). 

\begin{theorem}\label{thm:LjWe} Let $U$, $V$ be a pair of Lucas sequences with parameters $P$ and 
$Q$. Let $p\ge5$, $p\nmid Q$, be a prime whose rank of appearance $\rho$ is    
equal to $p-\e_p$, $\epsilon_p$ being $0$ or $\pm1$. 
Then, for all nonnegative integers $k$ and $\ell$, we have  
\begin{equation}\label{eq:Lju}
\binom{k\rho}{\ell\rho}_U\equiv\binom{k}{\ell}_{U'}(-1)^{\ell(k-\ell)\e_p}Q^{\ell(k-\ell)\rho(\rho-1)/2}\pmod {p^3},
\end{equation}
where $U'$ is the sequence $U_\rho\times U(V_\rho,Q^\rho)$. 
\end{theorem} 
\begin{proof} We only need a proof in case $k>\ell\ge1$. With convention (\ref{eq:conv}) we may write 
\begin{eqnarray*} \binom{k\rho}{\ell\rho}_U & = & \frac{U_{k\rho}U_{k\rho-1}\cdots U_{(k-\ell)\rho+1}}
{U_{\ell\rho}U_{\ell\rho-1}\cdots U_1}\\
& = & \frac{U_{k\rho}U_{(k-1)\rho}\cdots U_{(k-\ell+1)\rho}}{U_{\ell\rho} U_{(\ell-1)\rho}\cdots 
U_\rho}\cdot\frac{\prod_{i=k-\ell}^{k-1}
\prod_{t=1}^{\rho-1}U_{i\rho+t}}{\prod_{i=0}^{\ell-1}\prod_{t=1}^{\rho-1}U_{i\rho+t}}\\
& = & \binom{k}{\ell}_{U'}\cdot\frac{\prod_{i=k-\ell}^{k-1}\prod_{t=1}^{\rho-1}U_{i\rho+t}}
{\big(\prod_{t=1}^{\rho-1}U_t\big)^\ell}\cdot\frac{\big(\prod_{t=1}^{\rho-1}U_t\big)^\ell}
{\prod_{i=0}^{\ell-1}\prod_{t=1}^{\rho-1}U_{i\rho+t}}\\
& = &  \binom{k}{\ell}_{U'}\cdot\prod_{i=k-\ell}^{k-1}\binom{(i+1)\rho-1}{\rho-1}_U\cdot\bigg(
\prod_{i=0}^{\ell-1}\binom{(i+1)\rho-1}{\rho-1}_U\bigg)^{-1}\\
&\equiv &\binom{k}{\ell}_{U'}\cdot\binom{2\rho-1}{\rho-1}_U^{\sum_{i=k-\ell}^{k-1}i-\sum_{i=0}^{\ell-1}i}
\qquad(\text{ by Remark \ref{rem:pow} })\\
&=&\binom{k}{\ell}_{U'}\cdot\binom{2\rho-1}{\rho-1}_U^{\ell(k-\ell)}\pmod{p^3},
\end{eqnarray*}
yielding, by Theorem \ref{thm:N}, the theorem. 
\end{proof}

\begin{remark} If, in Theorem \ref{thm:LjWe}, $U_\rho\not=0$ then we might as well set $U'$ equal to 
$U(V_\rho,Q^\rho)$. 
\end{remark}

\begin{remark} If $U=U(2,1)$, then $U_t=t$ and $U'_t=pt$, or $U'_t=t$ by the above remark. 
Thus the theorem implies that 
$$
\binom{kp}{\ell p}\equiv\binom k\ell_{U'}=\binom k\ell\pmod{p^3},
$$ which is the classical congruence (\ref{eq:Viggo}) of Ljunggren et alii. For $U=U(1,-1)$     
and $\epsilon_p=\pm1$ we saw in Remark \ref{rem:kw} that $\epsilon_p=-(-1)^{\rho(\rho-1)/2}
=-Q^{\rho(\rho-1)/2}$ so that Theorem \ref{thm:LjWe} implies (\ref{eq:KW2}).
\end{remark}
 
 Since we took care of including all cases of Lucas sequences in our theorems, we provide 
an example of an application of Theorem \ref{thm:LjWe} to a degenerate Lucas sequence. 

\begin{example} Consider $U(2,2)$.  Its first terms are 
$$0,1,2,2,0,-4,-8,-8,0,16,32,32,0,\dots$$ 
So Theorem \ref{thm:LjWe} applies to $p=5$ since its rank is maximal and equal to $4$. Choose, 
say $k=3$ and $\ell=2$. By our extended definition of (\ref{eq:def}), we have $\binom{3}{2}_{U'}=1$ 
and $(-1)^{\ell(k-\ell)\e_p}Q^{\ell(k-\ell)\rho(\rho-1)/2}=2^{12}$. Computing 
$\binom{12}{8}_U$ we may verify the congruence modulo $125$, which in that case is an equality, since 
$$
\binom{12}{8}_U=\frac{U_{11}\cdot U_{10}\cdot U_9}{U_3\cdot U_2\cdot U_1}=
\frac{16\cdot32\cdot32}{2\cdot2\cdot1}=2^{12}.
$$
\end{example}

\section{Lucanomials $\binom{2\rho-1}{\rho-1}_U\pmod{p^5}$} 
\label{sec:4}

 The congruence of Wolstenholme has been studied to prime powers higher than the third. In particular, 
we have, for all primes $p\ge7$,
\begin{eqnarray}
\binom{2p-1}{p-1}&\equiv&1+p\sum_{0<t<p}\frac1{t}+p^2\sum_{0<s<t<p}\frac{1}{st}\pmod{p^5}\label{eq:i}\\
                 &\equiv&1+2p\sum_{0<t<p}\frac1{t}\pmod{p^5}\label{eq:ii}\\
                 &\equiv&1-p^2\sum_{0<t<p}\frac1{t^2}\pmod{p^5}\label{eq:iii}.
\end{eqnarray} 

 We will find congruences for the Lucanomial coefficients $\binom{2\rho-1}{\rho-1}_U$, valid for a general 
fundamental Lucas sequence $U$, modulo the fifth power of a prime 
of maximal rank $\rho$, which generalize the three congruences above. Expanding 
the binomial $\binom{2p-1}{p-1}$, as was done more generally for Lucanomials in the proof 
of Theorem \ref{thm:N}, one falls naturally on the congruence (\ref{eq:i}). This expansion 
appears, for instance, in the proof of Proposition 1 in \cite{Mes1}.  
Congruence (\ref{eq:ii}) is a special case of 
Theorem 3 of the paper \cite{Zh} and was known to hold for primes $p\ge5$ modulo $p^4$ much 
earlier, while congruence (\ref{eq:iii}) appears in \cite{McI}, p. 385. 

\medskip

 To complete the notation introduced in (\ref{eq:nota}) we define the symbols 
$\Sigma_{1,\nu}$ ($\nu=2$ or $3$), $\Sigma_{2,2}$, $\Sigma_{1,1,1}$, $\Sigma_{1,1,2}$ and $\Sigma_{1,1,1,1}$,  
respectively, as the sums 
$$
\sum_{s,\,t}\frac{V_sV_t^\nu}{U_sU_t^\nu},\;\sum_{s<t}\frac{V_s^2V_t^2}{U_s^2U_t^2},\;
\sum_{r<s<t}\frac{V_rV_sV_t}{U_rU_sU_t},\;\sum_{\substack{r<s,\\t\in(0,\rho)}}\frac{V_rV_sV_t^2}{U_rU_sU_t^2},\;
\sum_{q<r<s<t}\frac{V_qV_rV_sV_t}{U_qU_rU_sU_t},
$$  where in each sum $q$, $r$, $s$ and $t$ are distinct integers in the interval $(0,\rho)$ and 
$\rho$ is the rank of a prime $p$.  

\begin{lemma}\label{lem:sigmas} We have for all primes $p\ge7$ of maximal ranks
$$
\Sigma_{1,1,1}\equiv0\pmod{p^2} \text{ and } \Sigma_{1,1,1,1}\equiv\begin{cases}
0\pmod p,\quad\text{ if }\e_p=0 \text{ or }-1;\\
D^2\pmod p,\quad\text{ if }\e_p=1.
\end{cases}
$$
\end{lemma}
\begin{proof} We have the linear system 
\begin{eqnarray*} 
\Sigma_1^3-\Sigma_3 & = & 3\Sigma_{1,2}+6\Sigma_{1,1,1},\\
\Sigma_1\cdot\Sigma_{1,1} & = & \Sigma_{1,2}+3\Sigma_{1,1,1}. 
\end{eqnarray*}
Because $p^2$ divides both $\Sigma_1$ and $\Sigma_3$, $\Sigma_1^3-\Sigma_3$ and $\Sigma_1\cdot\Sigma_{1,1}$ 
are each $0\pmod {p^2}$. Since the determinant of the system is prime to $p$, 
$\Sigma_{1,2}$ and $\Sigma_{1,1,1}$ are both $0\pmod {p^2}$. 

 From Lemma \ref{lem:psquare} with $p>5$, which yields the values of $\Sigma_2$ and $\Sigma_4\pmod p$, we deduce that   
$$
\Sigma_{2,2}=\frac12\big[\Sigma_2^2-\Sigma_4\big]\equiv\begin{cases}0\pmod p,
\text { if }\e_p=0\text{ or }-1;\\
3D^2\pmod p,\text { if }\e_p=1.
\end{cases}
$$ 

 Now $\Sigma_{1,3}=\Sigma_1\cdot\Sigma_3-\Sigma_4\implies\Sigma_{1,3}\equiv-\Sigma_4\pmod p$. Moreover, 
$2\Sigma_{1,1,2}+2\Sigma_{2,2}+\Sigma_{1,3}=\Sigma_{1,2}\cdot\Sigma_1\equiv0\pmod p$.

 Thus, $\Sigma_{1,1,2}$ is $0\pmod p$, if $\e_p$ is $0$ or $-1$, and $\Sigma_{1,1,2}$ is 
$-4D^2\pmod p$, if $\e_p$ is $1$.   

Therefore, as $6\Sigma_{1,1,1,1}=\Sigma_{1,1}^2-\Sigma_{2,2}-2\Sigma_{1,1,2}$, we obtain, 
using Lemma \ref{lem:Sig}, the 
desired congruences for $\Sigma_{1,1,1,1}$.  
\end{proof}

 Our first theorem is a generalization of congruence (\ref{eq:i}). 

\begin{theorem}\label{thm:pcinq1} Let $(U,V)$ be a pair of Lucas sequence with parameters $P$ and $Q$. 
Let $p$ be a prime at least $7$ of maximal rank $\rho$ equal to $p-\e_p$. Then 
$$
\binom{2\rho-1}{\rho-1}_U\equiv
\bigg(\frac{V_\rho}{2}\bigg)^{\rho-1}\bigg[1+\frac{U_\rho}{V_\rho}\sum_{0<t<\rho}
\frac{V_t}{U_t}+\frac{U_\rho^2}{V_\rho^2}\sum_{0<s<t<\rho}\frac{V_sV_t}{U_sU_t}+R\bigg]\pmod{p^5},
$$
$$\text{ where }\;R=\frac{\epsilon_p(1+\epsilon_p)}2\frac{D^2U_\rho^4}{V_\rho^4}=
\begin{cases}0,\;\,\quad\qquad\text{ if }\e_p=0\text{ or }-1;\\
D^2U_\rho^4/V_\rho^4,\text{ if }\e_p=1.
\end{cases}
$$
\end{theorem} 
\begin{proof} Expanding the product $2^{\rho-1}\prod_{t=1}^{\rho-1}U_{\rho+t}=
\prod_{t=1}^{\rho-1}(V_\rho U_t+U_\rho V_t)$  
as we did early in the proof of Theorem \ref{thm:N}, but up to the fourth power of $U_\rho$, 
yields that $2^{\rho-1}\binom{2\rho-1}{\rho-1}_U$ is congruent to  
$$
V_\rho^{\rho-1}+V_\rho^{\rho-2}U_\rho\Sigma_1+
V_\rho^{\rho-3}U_\rho^2\Sigma_{1,1}+V_\rho^{\rho-4}U_\rho^3\Sigma_{1,1,1}+
V_\rho^{\rho-5}U_\rho^4\Sigma_{1,1,1,1}\pmod {p^5}.
$$ Applying the congruences obtained in Lemma \ref{lem:sigmas} to the last two terms 
of the above sum yields the theorem.
\end{proof}

 We now prove a congruence formula that generalizes (\ref{eq:ii}), but 
also generalizes Theorem \ref{thm:N} when $k=1$. The method of proof  
brings out the factor $(-1)^{\epsilon_p} Q^{\rho(\rho-1)/2}$ naturally.  
It is particularly appealing because it only contains two terms, no more than (\ref{eq:ii}), 
and is valid regardless of the values of the maximal rank $\rho$.    

\begin{theorem}\label{thm:pcinq2} Let $(U,V)$ be a pair of Lucas sequence with parameters $P$ and $Q$. 
Let $p$ be a prime at least $7$ of maximal rank $\rho$ equal to $p-\e_p$. Then
$$
\binom{2\rho-1}{\rho-1}_U\equiv
(-1)^{\e_p}Q^{\frac{\rho(\rho-1)}2}\bigg[1+2\frac{U_\rho}{V_\rho}\sum_{0<t<\rho}\frac{V_t}{U_t}
\bigg]\pmod{p^5}.
$$ 
\end{theorem}
\begin{proof} All unmarked sums and products are for $t$ running from $1$ to $\rho-1$. Note that 
$\prod U_t=\prod U_{\rho-t}$. Thus by (\ref{eq:6}) we may write 
\begin{eqnarray*}
2^{\rho-1}Q^{\sum t}\prod U_t & = & \prod2Q^tU_{\rho-t}=\prod(U_\rho V_t-V_\rho U_t)\\
& = & (-V_\rho)^{\rho-1}\prod\bigg(1-\frac{U_\rho}{V_\rho}\frac{V_t}{U_t}\bigg)\prod U_t.
\end{eqnarray*}
Therefore 
$$
(-1)^{\rho-1}Q^{\rho(\rho-1)/2}=\bigg(\frac{V_\rho}{2}\bigg)^{\rho-1}
\prod\bigg(1-\frac{U_\rho}{V_\rho}\frac{V_t}{U_t}\bigg),
$$ so that 
\begin{equation}\label{eq:red}
(-1)^{\rho-1}Q^{\rho(\rho-1)/2}\equiv\bigg(\frac{V_\rho}{2}\bigg)^{\rho-1}
\bigg(1-\frac{U_\rho}{V_\rho}\Sigma_1+
\frac{U_\rho^2}{V_\rho^2}\Sigma_{1,1}-\frac{U_\rho^3}{V_\rho^3}\Sigma_{1,1,1}+
\frac{U_\rho^4}{V_\rho}\Sigma_{1,1,1,1}\bigg)\pmod {p^5}.
\end{equation}
Note that from (\ref{eq:red}) we recover the congruence
\begin{equation}\label{eq:blanc}
(-1)^{\rho-1}Q^{\rho(\rho-1)/2}\equiv\bigg(\frac{V_\rho}{2}\bigg)^{\rho-1}\pmod {p^2}.
\end{equation} 
Subtracting the expansion in (\ref{eq:red}) from that of $\binom{2\rho-1}{\rho-1}_U$ obtained in the proof 
of Theorem \ref{thm:pcinq1}, we find that 
\begin{eqnarray*}
\binom{2\rho-1}{\rho-1}_U-(-1)^{\rho-1}Q^{\rho(\rho-1)/2}
&\equiv&\bigg(\frac{V_\rho}{2}\bigg)^{\rho-1}\bigg(2\frac{U_\rho}{V_\rho}\Sigma_1+
2\frac{U_\rho^3}{V_\rho^3}\Sigma_{1,1,1}\bigg)\\
&\equiv & 2\bigg(\frac{V_\rho}{2}\bigg)^{\rho-1}\frac{U_\rho}{V_\rho}\Sigma_1\pmod {p^5},
\end{eqnarray*} since $\Sigma_{1,1,1}$ is $0\pmod {p^2}$ by Lemma \ref{lem:sigmas}.
In the above congruence as $\frac{U_\rho}{V_\rho}\Sigma_1$ is $0\pmod {p^3}$ we may,  
by (\ref{eq:blanc}), replace $\big(\frac{V_\rho}{2}\big)^{\rho-1}$ by  
$(-1)^{\rho-1}Q^{\rho(\rho-1)/2}$ and deduce our theorem.  
\end{proof}

\begin{lemma}\label{lem:c}  Suppose $\nu$ is a nonnegative integer. 
Let $p\ge\nu+5$ be a prime of maximal rank, say $\rho$. Then 
$$
\sum_{0<t<\rho}\frac{4Q^t}{U_t^2}\frac{V_t^\nu}{U_t^\nu}=\Sigma_{\nu+2}-D\Sigma_\nu\equiv
\begin{cases} 0\pmod{p^2},\;\text{ if }\nu\text{ is odd};\\
0\pmod p,\;\;\text{ if }\nu\text{ is even.}
\end{cases}
$$
\end{lemma}
\begin{proof} We have 
$$
\sum_{0<t<\rho}\frac{4Q^t}{U_t^2}\frac{V_t^\nu}{U_t^\nu}=\sum_{0<t<\rho}
\frac{(V_t^2-DU_t^2)}{U_t^2}\frac{V_t^\nu}{U_t^\nu}=\Sigma_{\nu+2}-D\Sigma_\nu. 
$$
If $\nu$ is odd, then, $p\ge\nu+5$ implies, by Lemma \ref{lem:psquare}, 
that both $\Sigma_\nu$ and $\Sigma_{\nu+2}$ are $0\pmod{p^2}$. If $\nu$ is even, 
then both $\Sigma_{\nu+2}$ and $D\Sigma_\nu$ are $0\pmod p$, when $\rho$ is $p$ or $p+1$,  
by Lemma \ref{lem:psquare}. If $\rho$ is $p-1$, then by the same lemma 
$\Sigma_{\nu+2}-D\Sigma_\nu\equiv-2D^{\frac{\nu+2}2}-D(-2D^{\nu/2})\equiv0\pmod p$. 
\end{proof}

\begin{lemma}\label{lem:s1} We have for all primes $p\ge7$ of maximal rank $\rho$
$$
-2\Sigma_1\equiv\frac{U_\rho}{V_\rho}\sum_{0<t<\rho}\frac{4Q^t}{U_t^2}
\pmod {p^4}.
$$
\end{lemma}
\begin{proof} All sums are over an index $t$ running from $1$ to $\rho-1$. 
\begin{eqnarray*}
-2\Sigma_1&=&-\sum\bigg(\frac{V_t}{U_t}+\frac{V_{\rho-t}}{U_{\rho-t}}\bigg)=
-2U_\rho\sum\frac1{U_tU_{\rho-t}}, \text{ by }(\ref{eq:1}),\\
&=&-2U_\rho\sum\frac{2Q^t}{U_t(U_\rho V_t-U_tV_\rho)}, \text{ using }(\ref{eq:6}),\\
&=&2\frac{U_\rho}{V_\rho}\sum \frac{2Q^t}{U_t^2\big[1-\frac{V_t}{U_t}
\frac{U_\rho}{V_\rho}\big]}\\
&\equiv&\frac{U_\rho}{V_\rho}\sum\frac{4Q^t}{U_t^2}\bigg[1+\frac{V_t}{U_t}\frac{U_\rho}{V_\rho}
+\frac{V_t^2}{U_t^2}\frac{U_\rho^2}{V_\rho^2}\bigg]\pmod{p^4},
\end{eqnarray*}
because, by Lemma \ref{lem:c}, $U_\rho^{\nu+1}\sum\frac{4Q^t}{U_t^2}
\frac{V_t^\nu}{U_t^\nu}$ is $0\pmod{p^4}$, for $\nu=1$ and $\nu=2$, if $p\ge7$. 
\end{proof} 

 From Theorem \ref{thm:pcinq2}, it is not difficult to reach a third theorem that 
generalizes (\ref{eq:iii}).

\begin{theorem}\label{thm:pcinq3} Let $(U,V)$ be a pair of Lucas sequence with parameters $P$ and $Q$. 
Let $p$ be a prime at least $7$ of maximal rank $\rho$ equal to $p-\e_p$. Then
$$
\binom{2\rho-1}{\rho-1}_U\equiv
(-1)^{\e_p}Q^{\frac{\rho(\rho-1)}2}\bigg[1-4\frac{U_\rho^2}{V_\rho^2}\sum_{0<t<\rho}\frac{Q^t}{U_t^2}
\bigg]\pmod{p^5}.
$$ 
\end{theorem}
\begin{proof} In the congruence for the Lucanomial $\binom{2\rho-1}{\rho-1}_U$ of Theorem \ref{thm:pcinq2}  
we may replace $2\frac{U_\rho}{V_\rho}\Sigma_1$ by $-\frac{U_\rho^2}{V_\rho^2}\sum\frac{4Q^t}{U_t^2}$ 
since by Lemma \ref{lem:s1} the two expressions are congruent modulo ${p^5}$. 
\end{proof}

\begin{remark} In stating Theorem \ref{thm:pcinq3} we chose the expression 
$-4\frac{U_\rho^2}{V_\rho^2}\sum\frac{Q^t}{U_t^2}$ 
rather than $-\frac{U_\rho^2}{V_\rho^2}\Sigma_2+\frac{U_\rho^2}{V_\rho^2}(\rho-1)D$ 
because it contains only one term; that term is $0\pmod{p^3}$ and it reduces to 
$-p^2\sum\frac1{t^2}$ for $U=U(2,1)$.  
\end{remark}

\begin{lemma}\label{lem:go} We have for all primes $p\ge7$ of maximal rank $\rho$
$$
\frac{U_\rho}{V_\rho}\Sigma_1\equiv\frac{U_\rho^2}{V_\rho^2}\Sigma_{1,1}-\frac12
\frac{U_\rho^2}{V_\rho^2}(\rho-1)D\pmod {p^5}.
$$
\end{lemma}
\begin{proof} By Lemma \ref{lem:s1}, we see that 
$$
\frac{U_\rho}{V_\rho}\Sigma_1\equiv-\frac12\frac{U_\rho^2}{V_\rho^2}\sum_{0<t<\rho}\frac{4Q^t}{U_t^2}\pmod {p^5}.
$$ By Lemma \ref{lem:c},  
$$
\sum_{0<t<\rho}\frac{4Q^t}{U_t^2}=\Sigma_2-D(\rho-1).
$$ Thus, as $\Sigma_2=\Sigma_1^2-2\Sigma_{1,1}\equiv-2\Sigma_{1,1}\pmod {p^4}$, 
the lemma follows. 
\end{proof}

 By using Lemma \ref{lem:go} and Theorem \ref{thm:pcinq2} we obtain another generalization of (\ref{eq:i}) 
slightly different from that given in Theorem \ref{thm:pcinq1}, which we now state. 

\begin{theorem}\label{thm:pcinq4} Let $(U,V)$ be a pair of Lucas sequences with parameters $P$ and $Q$. 
Let $p$ be a prime at least $7$ of maximal rank $\rho$ equal to $p-\e_p$. Then
$\binom{2\rho-1}{\rho-1}_U$ is congruent to
$$
(-1)^{\e_p}Q^{\frac{\rho(\rho-1)}2}\bigg[1+\frac{U_\rho}{V_\rho}\sum_{0<t<\rho}\frac{V_t}{U_t}+
\frac{U_\rho^2}{V_\rho^2}\sum_{0<s<t<\rho}\frac{V_sV_t}{U_sU_t}-\frac12D
\frac{U_\rho^2}{V_\rho^2}(\rho-1)\bigg]\pmod{p^5}.
$$
\end{theorem}

 We end the paper with a congruence for $\binom{2\rho-1}{\rho-1}_U$ modulo ${p^6}$. It generalizes 
Theorem 2.4 of \cite{Tau} which says that   
$$
\binom{2p-1}{p-1}\equiv1+2p\sum_{0<t<p}\frac1t+\frac{2p^3}3\sum_{0<t<p}\frac1{t^3}\pmod{p^6},
$$ 
for all primes $p\ge7$, and also generalizes our Theorem \ref{thm:pcinq2}. 

\begin{theorem} Let $(U,V)$ be a pair of Lucas sequences with parameters $P$ and $Q$. 
Let $p$ be a prime at least $7$ of maximal rank $\rho$. Then 
$$
\binom{2\rho-1}{\rho-1}_U\equiv(-1)^{\rho-1}Q^{\frac{\rho(\rho-1)}2}\bigg[1+2\frac{U_\rho}
{V_\rho}\sum_{0<t<p}\frac{V_t}{U_t}+
\frac23\frac{U_\rho^3}{V_\rho^3}\sum_{0<t<p}\frac{V_t^3}{U_t^3}\bigg]\pmod{p^6}.
$$
\end{theorem} 
\begin{proof} We proceed as in Lemma \ref{lem:sigmas} to show that $\Sigma_{1,1,1,1,1}\equiv0\pmod{p}$ 
(in fact $0$ modulo $p^2$). 
First we extend the definitions made before Lemma \ref{lem:sigmas} to define analogously 
the sums $\Sigma_{1,4}$, $\Sigma_{1,1,3}$, $\Sigma_{2,3}$, $\Sigma_{1,2,2}$ and $\Sigma_{1,1,1,2}$.  
The expressions $\Sigma_1\cdot\Sigma_4-\Sigma_5$, $\Sigma_3\cdot\Sigma_{1,1}$, $\Sigma_1\cdot\Sigma_{1,3}$ 
and $\Sigma_1\cdot\Sigma_{2,2}$ are all $0\pmod{p^2}$, so we deduce, successively, that the  
sums $\Sigma_{1,4}$, $\Sigma_{1,1,3}$, $\Sigma_{2,3}$ and $\Sigma_{1,2,2}$ are each $0\pmod{p^2}$. 
Therefore, modulo $p^2$,  the two expressions $\Sigma_1\cdot\Sigma_{1,1,1,1}$ 
and $\Sigma_1^5-\Sigma_5$ are linear combinations of $\Sigma_{1,1,1,1,1}$ and $\Sigma_{1,1,1,2}$.  
Because these two expressions are each $0\pmod{p^2}$ we deduce that  
$\Sigma_{1,1,1,1,1}\equiv0\pmod{p^2}$.

\medskip

Since $\Sigma_{1,1,1,1,1}$ is 
$0\pmod p$, both the congruence for $\binom{2\rho-1}{\rho-1}_U$, derived from the proof 
of Theorem \ref{thm:pcinq1}, and congruence (\ref{eq:red}) remain valid when we raise  
the modulus from $p^5$ to $p^6$. 
Hence, 
\begin{equation}\label{eq:vent}
\binom{2\rho-1}{\rho-1}_U-(-1)^{\rho-1}Q^{\rho(\rho-1)/2}
\equiv\bigg(\frac{V_\rho}{2}\bigg)^{\rho-1}\bigg(2\frac{U_\rho}{V_\rho}\Sigma_1+
2\frac{U_\rho^3}{V_\rho^3}\Sigma_{1,1,1}\bigg)\pmod {p^6}. 
\end{equation}
Suppose first that $\epsilon_p=-1$ or $\epsilon_p=0$. Then, as $\Sigma_{1,1}\equiv0\pmod p$, 
we find that (\ref{eq:blanc}) is valid modulo $p^3$. Thus, we may replace $(V_\rho/2)^{\rho-1}$ 
in (\ref{eq:vent}) by $(-1)^{\rho-1}Q^{\rho(\rho-1)/2}$ and obtain that 
\begin{equation}\label{eq:eau}
\binom{2\rho-1}{\rho-1}_U\equiv(-1)^{\rho-1}Q^{\rho(\rho-1)/2}
\bigg(1+2\frac{U_\rho}{V_\rho}\Sigma_1+
2\frac{U_\rho^3}{V_\rho^3}\Sigma_{1,1,1}\bigg)\pmod {p^6}. 
\end{equation} Looking at the linear system at the start of the proof of Lemma \ref{lem:sigmas} modulo 
$p^3$ we find the system of congruences
\begin{eqnarray*} 
3\Sigma_{1,2}+6\Sigma_{1,1,1} & \equiv & -\Sigma_3,\\
\Sigma_{1,2}+3\Sigma_{1,1,1} & \equiv &0. 
\end{eqnarray*}
Solving for $\Sigma_{1,1,1}$, we see that $\Sigma_{1,1,1}\equiv\frac{\Sigma_3}{3}\pmod{p^3}$, 
which inserted in congruence (\ref{eq:eau}) yields the theorem. 

Suppose now $\epsilon_p=1$ so that congruence (\ref{eq:blanc}), when the modulus is increased to $p^3$, 
becomes 
$$
(-1)^{\rho-1}Q^{\rho(\rho-1)/2}\equiv(V_\rho/2)^{\rho-1}(1+DU_\rho^2/V_\rho^2)
\pmod {p^3}.
$$
Thus we may replace $(V_\rho/2)^{\rho-1}$ in (\ref{eq:vent}) by $(-1)^{\rho-1}Q^{\rho(\rho-1)/2}
(1-DU_\rho^2/V_\rho^2)$, multiply out the resulting expression and remove the term 
in $U_\rho^5\Sigma_{1,1,1}$ which is $0\pmod {p^7}$ to find that 
$$
\binom{2\rho-1}{\rho-1}_U\equiv(-1)^{\rho-1}Q^{\rho(\rho-1)/2}\bigg(1+2\frac{U_\rho}{V_\rho}\Sigma_1
+2\frac{U_\rho^3}{V_\rho^3}(\Sigma_{1,1,1}-D\Sigma_1)\bigg)\pmod {p^6}.
$$
Because $\Sigma_1$ is $0\pmod{p^2}$  and $\Sigma_{1,1}\equiv D\pmod{p}$, the linear system 
of Lemma \ref{lem:sigmas} taken modulo $p^3$ is 
\begin{eqnarray*} 
3\Sigma_{1,2}+6\Sigma_{1,1,1} & \equiv & -\Sigma_3,\\
\Sigma_{1,2}+3\Sigma_{1,1,1} & \equiv & D\Sigma_1. 
\end{eqnarray*}
Solving for $\Sigma_{1,1,1}$ yields $\Sigma_{1,1,1}\equiv D\Sigma_1+\Sigma_3/3$ and  
the theorem holds.  
\end{proof} 

\section{Appendix on the integrality of Lucanomials}

 The question of the integrality of Lucanomials has appeared in various places, but 
we want to formally prove that with convention (\ref{eq:conv}) they are integral  
in full generality.  

\begin{prop} Let $U=(U_n)$ be a Lucas sequence with parameters $P$ and $Q$. With 
the adoption of convention (\ref{eq:conv}) the Lucanomial coefficients 
$\binom{m}{n}_U$ are rational integers for all nonnegative integers $m$ and $n$. 
\end{prop}
\begin{proof} If all $U_n$, $n>0$, are nonzero then the frequently used induction argument 
(see \cite{HuSu}, Lemma 1; or \cite{Ben}) based on the general Lucas identity 
$U_{n+1}U_{m-n}-QU_nU_{m-n-1}=U_m$ works fine. 
(The induction is on $m$. So one proves the integrality of the Lucanomial $\binom{m}{n}_U$ 
for $m>n\ge1$ by observing that
\begin{eqnarray*}
U_{n+1}\binom{m-1}{n}_U-QU_{m-n-1}\binom{m-1}{n-1}_U & = &\\
\bigg(U_{n+1}\frac{U_{m-n}}{U_n}-QU_{m-n-1}\bigg)\cdot\binom{m-1}{n-1}_U & = &\\
\frac{U_m}{U_n}\cdot\binom{m-1}{n-1}_U=\binom{m}{n}_U&,&
\end{eqnarray*} 
completing the induction.)
If some term $U_n$, $n\ge1$, is $0$ then $U$ is degenerate and, as we saw early in 
Section \ref{sec:2}, $\rho(\infty)\in\{2,3,4,6\}$, where $\rho(\infty)$ is the least 
positive integer $t$ such that $U_t=0$. Note that we may always assume $m\ge2n$. Thus the 
Lucanomial $\binom{m}{n}_U$ is the quotient of a product of $n$ consecutive $U$ terms 
of indices all larger than $n$ divided by $U_nU_{n-1}\cdots U_1$. 
If $\rho(\infty)=2$, i.e., $U_2=P=0$, then 
$U_{2k+1}=(-1)^kQ^k$ and $U_{2k}=0$, ($k\ge0$). Then $\binom{m}{n}_U$ is up to sign 
a positive power of $Q$. If $\rho(\infty)=3$, then, as $U_3=P^2-Q$, the first few terms of $U$ are 
$0,1,P,0,-P^3,-P^4,0,P^6,P^7,0,\cdots$. So $|U_t|=P^{t-1}$ if $3\nmid t$. 
If $\rho(\infty)=4$, then, as $U_4=P^3-2PQ$ and $P\not=0$, 
$P^2=2Q$ and we see that $|U_t|=2^{\lfloor t/2\rfloor}(P')^{t-1}$ if $4\nmid t$, 
where $P=2P'$. Omitting the $0$ terms when $4\mid t$, powers of $2$ and $P'$ 
in $U_t$ are nondecreasing functions of $t$. A similar result holds for $\rho(\infty)$ equal to 
$6$ when $P^2=3Q$ and, omitting terms divisible by $6$, powers of $3$ and of $P'$ in $U_t$ 
are nondecreasing functions of $t$, where in this case $P=3P'$. The integrality of the Lucanomials 
follows readily.  
\end{proof}

-----------------------------------------------------------------------------------

2010 {\it Mathematics Subject Classification}: 11A07, 11B65, 11B39.

{\it Keywords}: generalized binomial coefficients, Wolstenholme's congruence, Lucas sequences, rank of appearance.


\begin{thebibliography}{99}

\bibitem{Bab} C. Babbage, Demonstration of a theorem relating to prime numbers, 
{\it Edinburgh Philosophical J.}, {\bf 1} (1819), 46--49. 

\bibitem{Bai} D. F. Bailey, Two $p^3$ variations of Lucas' theorem, {\it J. Number Theory}, 
{\bf 35} (1990), no. 2, 208--215.

\bibitem{Ba1} C. Ballot, On a congruence of Kimball and Webb involving Lucas sequences, 
{\it J. Integer Seq.}, {\bf 17} (2014), Article 14.1.3.   

\bibitem{Ba2} C. Ballot, Lucas sequences with cyclotomic root field, 
{\it Dissertationes Math.}, {\bf 490} (2013), 92 pp.

\bibitem{Ba3} C. Ballot, A further generalization of a congruence of Wolstenholme, 
{\it J. Integer Seq.}, {\bf 15} (2012), Article 12.8.6.   

\bibitem{Ben} A. Benjamin and S. Plott, A combinatorial approach to Fibonomial coefficients,  
{\it Fibonacci Quart.}, {\bf 46/47} (2008/09), no. 1, 7--9. 

\bibitem{Bru} V. Brun, J. O. Stubban, J. E. Fjeldstad, R. Tambs Lyche, K. E. Aubert, 
W. Ljunggren, E. Jacobsthal. On the divisibility of the difference between 
two binomial coefficients. Den 11te Skandinaviske Matematikerkongress, Trondheim, 1949, pp. 42--54. 
{\it Johan Grundt Tanums Forlag, Oslo}, 1952.  

\bibitem{Fr} R. D. Fray, Congruence properties of ordinary and q-binomial coefficients.
{\it Duke Math. J.}, {\bf 34}, (1967) 467--480. 

\bibitem{Glai1} J. W. L. Glaisher, Congruences relating to the sums of products of the first $n$ 
numbers and to other sums of products, {\it Q. J. Math.} {\bf 31} (1900), 1--35. 

\bibitem{Glai2} J. W. L. Glaisher, On the residues of the sums of products of the first $p-1$ 
numbers, and their powers, to modulus $p^2$ or $p^3$, {\it Q. J. Math.} {\bf 31} (1900), 321--353. 

\bibitem{Gou} H. W. Gould, The bracket function and Fonten\'e-Ward generalized binomial 
coefficients with application to Fibonomial coefficients, {\it Fibonacci Quart.\/}, {\bf 7.1} 
(1969) 23--40, 55.

\bibitem{Hoo} C. Hooley, On Artin's conjecture, {\it J. Reine Angew. Math.} 
{\bf  225} (1967), 209--220.

\bibitem{HuSu} H. Hu and Z-W Sun, An extension of Lucas' theorem,  
{\it Proc. Amer. Math. Soc.} {\bf 129} (2001), no. 12, 3471--3478.  

\bibitem{KW1} W. Kimball and W. Webb, A congruence for Fibonomial coefficients 
modulo $p^3$, {\it Fibonacci Quart.\/}, {\bf 33} (1995) 290--297.  

\bibitem{KW2} W. Kimball and W. Webb, Some generalizations of Wolstenholme's 
theorem, {\it Applications of Fibonacci Numbers} {\bf 8} (Rochester, NY, 1998), 
Kluwer Acad. Publ., Dordrecht, (1999), 213--18. 

\bibitem{KnWi} D. Knuth and H. Wilf, The power of a prime that divides a generalized binomial coefficient,  
{\it J. Reine Angew. Math.} {\bf 396} (1989), 212--219.

\bibitem{McI} R. J. McIntosh, On the converse of Wolstenholme's Theorem, {\it Acta Arith.\/}, 
{\bf 71} (1995), 381--389. 

\bibitem{Mes} R. Me\v{s}trovi\`c, Wolstenholme's theorem: Its generalizations and extensions 
in the last hundred and fifty years (1862-2012), preprint {\tt http://arxiv.org/abs/1111.3057v2[math.NT]}. 

\bibitem{Mes1} R. Me\v{s}trovi\`c, Congruences for Wolstenholme primes, 
preprint {\tt http://arxiv.org/abs/1108.4178[math.NT]}. 

\bibitem{Pan} H. Pan, A generalization of Wolstenholme's harmonic series congruence, 
{\it Rocky Mountain J. Math.\/}, {\bf 38} (2008), 1263--1269. 
 
\bibitem{Stan} R. Stanley, {\it Enumerative combinatorics}, Volume 1. Second edition. 
Cambridge Studies in Advanced Mathematics, 49. Cambridge University Press, Cambridge, (2012).  

\bibitem{Stra} A. Straub, A q-analog of Ljunggren's binomial congruence, 23rd International Conference 
on Formal Power Series and Algebraic Combinatorics (FPSAC 2011), 897--902, 
{\it Discrete Math. Theor. Comput. Sci. Proc.}, AO, Assoc. Discrete Math. Theor. Comput. Sci., Nancy, 2011

\bibitem{Ros} H. Roskam, A quadratic analogue of Artin's conjecture on primitive roots. 
{\it J. Number Theory} {\bf 81} (2000), no. 1, 93--109.

\bibitem{Ros1} H. Roskam, Erratum: ``A quadratic analogue of Artin's conjecture on primitive roots'' 
[J. Number Theory 81 (2000), no. 1, 93--109. {\it J. Number Theory} {\bf 85} (2000), no. 1, 108. 

\bibitem{Tau} R. Tauraso, More congruences for central binomial coefficients. 
{\it J. Number Theory}, {\bf 130} (2010), no. 12, 2639--2649.

\bibitem{Wa} M. Ward, Note on divisibility sequences. {\it Bull. Amer. Math. Soc.} 
{\bf 42} (1936), no. 12, 843--845. 

\bibitem{Wi} H. C. Williams, {\it \'Edouard Lucas and primality testing}, 
Wiley, Canadian Math. Soc. Series of Monographs and Advanced Texts, (1998). 

\bibitem{Wo} J. Wolstenholme, On certain properties of prime numbers, 
{\it Q. J. Pure Appl. Math.\/}, {\bf 5} (1862), 35--39. 

\bibitem{Zh} J. Zhao, Bernoulli numbers, Wolstenholme's theorem, and $p^5$ variations of Lucas' theorem. 
{\it J. Number Theory}, {\bf 123} (2007), no. 1, 18--26.

\end{thebibliography}
\end{document}